\documentclass[12 pt]{amsart}

\usepackage{amssymb}
\usepackage{amsbsy}
\usepackage{amscd}
\usepackage{amsmath}
\usepackage{amsthm}

\usepackage{bm}

\newtheorem{theorem}{Theorem}[section]
\newtheorem{defn}[theorem]{Definition}
\newtheorem{lem}[theorem]{Lemma}
\newtheorem{prop}[theorem]{Proposition}
\newtheorem{rem}[theorem]{Remark}
\newtheorem{example}[theorem]{Example}
\newtheorem{corollary}[theorem]{Corollary}
\newtheorem{pbm}[theorem]{Problem}
\newtheorem{conj}[theorem]{Conjecture}

\newcommand{\codim}{\operatorname{codim}}
\newcommand{\Coker}{\operatorname{Coker}}
\newcommand{\Der}{\operatorname{Der}}
\newcommand{\Hilb}{\operatorname{Hilb}}
\newcommand{\Image}{\operatorname{Im}}
\newcommand{\HOM}{\operatorname{{\mathcal H}om}}

\newcommand{\pdeg}{\operatorname{pdeg}}

\newcommand{\Proj}{\operatorname{Proj}}
\newcommand{\Real}{\operatorname{Re}}

\newcommand{\A}{\mathcal{A}}
\newcommand{\calC}{\mathcal{C}}
\newcommand{\calE}{\mathcal{E}}
\newcommand{\calF}{\mathcal{F}}

\newcommand{\calM}{\mathcal{M}}
\newcommand{\calO}{\mathcal{O}}

\newcommand{\bC}{\mathbb{C}}
\newcommand{\bF}{\mathbb{F}}

\newcommand{\bP}{\mathbb{P}}
\newcommand{\bQ}{\mathbb{Q}}
\newcommand{\bR}{\mathbb{R}}
\newcommand{\bZ}{\mathbb{Z}}

\newcommand{\m}{{\bf m}}

\author[M. Yoshinaga]{Masahiko Yoshinaga}
\title[Free Arrangements]{Freeness of hyperplane arrangements and related topics}
\address{Department of Mathematics, Hokkaido University, 
N 10 W 8 Sapporo 060-0810, Japan}
\curraddr{}
\email{yoshinaga@math.sci.hokudai.ac.jp}

\begin{document}

\maketitle

\begin{abstract}

These are the expanded notes of the lecture by the author in 
``Arrangements in Pyr\'en\'ees'', June 2012. 
We are discussing relations of freeness 
and splitting problems of vector bundles, several techniques 
proving freeness of hyperplane arrangements, 
K. Saito's theory of primitive derivations for 
Coxeter arrangements, their application to combinatorial problems 
and related conjectures. 
\end{abstract}

\addtocounter{section}{-1}
\section{Introduction}


Roughly speaking, there are two kind of objects in mathematics: 
general objects and specialized objects. In the study of 
general objects, individual objects are 
not so important, the totality of 
general objects is rather interesting (e.g. stable algebraic curves 
and moduli spaces). On the other hand, specialized objects are 
isolated, tend to be studied individually. 

Let us fix a manifold (algebraic, complex analytic, whatever) $X$. 
Then the divisors on $X$ are general objects. In 1970's Kyoji Saito 
\cite{sai-log} introduced the notion of {\em free divisors} with the 
motivation to compute Gauss-Manin connections for universal 
unfolding of isolated singularities. It was proved that the 
discriminant in the parameter space of the universal unfolding 
is a free divisor. 
Free divisors are specialized objects. 
The discriminant for a simple singularity is 
obtained as a quotient of the union of hyperplanes of finite 
reflection group, which implies that the union of reflecting 
hyperplanes (Coxeter arrangement) is also a free divisor (free 
arrangement). He also studied Coxeter arrangements in terms 
of invariant theory and found deep structures related to freeness. 
This has made deep impact on combinatorics of Coxeter arrangements 
(which is summarized in \S\ref{sec:prim}). Subsequently 
Terao developed basic techniques and laid the foundations of 
the theory of free arrangements. Now this becomes a rich area 
which is related to combinatorics and algebraic geometry. 

The purpose of this article is to survey the aspects of 
free arrangements. 
\S\ref{sec:spl} is devoted to describe techniques proving 
freeness. 
In early days, the freeness of arrangements was studied 
mainly from algebraic and combinatorial view point. 
It was pointed out by Silvotti \cite{sil-poi} and Schenck 
\cite{sch-rk2, mus-sch} that the freeness 
is equivalent to splitting of a reflexive 
sheaf on the projective space $\bP^n$ into sum of line bundles 
(``splitting problem''). This point of view has been a source of 
ideas of several recent studies on freeness of arrangements. 
The importance of multiarrangements emerged in these researches, and 
the general theory of free multiarrangements has been developed in the 
last decade. 
We are trying to depict them in \S\ref{sec:spl}. 

In \S\ref{sec:prim}, we summarize the theory of primitive derivation 
for Coxeter arrangements. 
I recommend \cite{sai-orb} for full details. 

\S\ref{sec:appl} is devoted to the problems concerning 
truncated affine Weyl arrangements. As an application of 
results in previous sections, it is proved that 
the so-called extended Catalan and extended Shi arrangements 
are free, which has also combinatorial consequences 
via Terao's factorization theorem \cite{ter-fact} and Solomon-Terao's 
formula \cite{st-stf}. 
This settled the conjecture by Edelman and Reiner \cite{ede-rei}. 
We also try to convince that some 
open problems (including ``Riemann Hypothesis'' by Postnikov and 
Stanley \cite{ps-def}) seems to be related to the algebraic structures 
studied in \S\ref{sec:prim}. 

\medskip

The author would like to thank Takuro Abe, Daniele Faenzi and 
Michele Torielli 
for comments and sharing unpublished ideas concerning this article. 
He also thanks organizers of the school 
``Arrangements in Pyr\'en\'ees'', Pau, June 2012.

\section{Splitting v.s. Freeness}
\label{sec:spl}

In this section, we discuss relations of freeness of divisors 
(especially hyperplane arrangements) and splitting problems of 
vector bundles. We emphasize parallelism and subtle 
differences. 

\subsection{Splitting problems}

First let us recall the correspondence between graded modules 
and coherent sheaves on the projective space. 
(Basic reference is \cite[Chap II \S5]{har-ag}) 
Let $S=\bC[x_1, x_2, \dots, x_\ell]$ be the polynomial ring 
and $\bP_\bC^{\ell-1}=\Proj S$ the projective $(\ell-1)$-space (denote 
$\bP^{\ell-1}$ for simplicity). $\bP^{\ell-1}$ 
is covered by open subsets $U_{x_i}$ ($i=1, \dots, \ell$), 
where $U_{x_i}$ is an open subset defined by $\{x_i\neq 0\}$. 
Let $M$ be a graded $S$-module. 
Then $M$ induces a sheaf $\widetilde{M}$ on $\bP^{\ell-1}$, with sections 
$$
\Gamma(U_{x_i}, \widetilde{M})=(M_{x_i})_0, 
$$
where $M_{x_i}=M\otimes_S S[\frac{1}{x_i}]$ is the localization by 
$x_i$ and $(-)_d$ denotes the degree $d$ component of the graded module. 
For $k\in\bZ$, 
define the graded module $M[k]$ by shifting degrees by $k$, namely, 
$M[k]_d=M_{d+k}$. Denote $\calO=\widetilde{S}$. 
The sheaf $\widetilde{S[k]}$ is a rank one module over $\calO$, which 
is denoted by $\calO(k)$. 

Using the natural map 
$\Gamma(\bP^{\ell-1}, \calE)\times\Gamma(\bP^{\ell-1}, \calF)\longrightarrow 
\Gamma(\bP^{\ell-1}, \calE\otimes\calF)$, we can define a graded ring 
structure on $\Gamma_*(\calO):=
\bigoplus_{d\in\bZ}\Gamma(\bP^{\ell-1}, \calO(d))$, 
which is isomorphic to $S$. 
More generally, for any sheaf ($\calO$-module) $\calF$ on $\bP^{\ell-1}$, 
$$
\Gamma_*(\calF):=
\bigoplus_{d\in\bZ}\Gamma(\bP^n, \calF\otimes\calO(d))
$$
has a graded $S$-module structure. 
For a graded $S$-module $M$, $\Gamma_*(\widetilde{M})$ is expressed 
as 
\begin{equation*}
\begin{split}
\Gamma_*(\widetilde{M})&=
\{
(f_1, \dots, f_\ell)\mid 
f_i\in M_{x_i},\ 
f_i=f_j \mbox{ in }M_{x_ix_j}\}\\
&=
\{f\in M_{x_1x_2\dots x_\ell}\mid 
\exists N\gg 0, x_i^N f\in M, \forall i=1, \dots, \ell
\}. 
\end{split}
\end{equation*}
Hence there is a natural homomorphism 
$\alpha:M\longrightarrow\Gamma_*(\widetilde{M})$. 
The above map $\alpha$ is not necessarily isomorphic. 

\begin{defn}
\normalfont
A sheaf of $\calO$-modules $\calF$ on $\bP^n$ 
is said to be {\em splitting} if there exist integers 
$d_1, \dots, d_r\in\bZ$ such that 
$$
\calF\simeq \calO(d_1)\oplus\dots\oplus\calO(d_r). 
$$
(Note that if we pose $d_1\geq d_2\geq\dots\geq d_r$, the degrees 
are uniquely determined.) 
\end{defn}

Let $\calE$ be an $\calO$-module. Denote by $\calE^\vee=
\HOM_\calO(\calE, \calO)$ the dual module of $\calE$. An $\calO$-module 
$\calE$ is called {\em reflexive} if the natural map 
$\calE\longrightarrow\calE^{\vee\vee}$ is an isomorphism. 
$\calE$ is called a {\em vector bundle} if it is locally free. 

A torsion free $\calO$-module on $\bP^1$ is always splitting. 
\begin{theorem}
(Grothendieck's splitting theorem) 
Let $\calE$ be a vector bundle on $\bP^1$. Then $\calE$ is 
splitting. 
\end{theorem}
A vector bundle $\calE$ on $\bP^n$, with $n\geq 2$ is non-splitting 
in general. For example, 
the tangent bundle $T_{\bP^n}$ is irreducible 
rank $n$ vector bundle on $\bP^n$ for $n\geq 2$, i.e. not 
a sum of proper sub-bundles. 

Let $\calE$ be a torsion free sheaf. 
Let $H$ be a hyperplane defined by a linear form $\alpha$. 
Since $\alpha\in\Gamma(\bP^n, \calO(1))$, we have the following 
short exact sequence 
\begin{equation}
\label{eq:short}
0\longrightarrow
\calE(-1)
\stackrel{\alpha\cdot}{\longrightarrow}
\calE
\longrightarrow
\calE|_H
\longrightarrow
0. 
\end{equation}
The short exact sequence (\ref{eq:short}) plays a crucial 
role in splitting problems. 

Let $\calE$ be a rank $r$ vector bundle on $\bP^n$. 
Then 
$\det\calE:=\bigwedge^r\calE$ is a line bundle and is called 
the {\em determinant bundle}. {\em The first Chern number} 
of $\calE$ 
is the integer 
$c_1\in\bZ$ 
satisfying $\det\calE=\calO(c_1)$. 

\begin{prop}
\label{prop:easy}
Let $\calE$ be a rank $r$ vector bundle on $\bP^n$ with $n\geq 2$. 
\begin{itemize}
\item[$(i)$] 
Let $\delta_i\in\Gamma(\calE\otimes\calO(-d_i))$ for certain 
$d_i\in\bZ,\ i=1, \dots, r$. Assume that $\delta_1, \dots, \delta_r$ are 
linearly independent over rational function field (or equivalently, 
$\delta_1\wedge\dots\wedge\delta_r\in
\Gamma(\det\calE\otimes\calO(-d_1-\dots-d_r))$ is non-zero) and 
$\sum_{i=1}^rd_i=c_1(\calE)$. 
Then $\calE$ is splitting and $\calE=\bigoplus_{i=1}^r
\calO(d_i)$. 
\item[$(ii)$] 
Let $H\subset\bP^n$ be a hyperplane. If the restriction 
$\calE|_H$ to $H$ is splitting and the induced map 
$$
\Gamma_*(\calE)\longrightarrow
\Gamma_*(\calE|_H)
$$
is surjective, then $\calE$ is also splitting. 
\end{itemize}
\end{prop}
\begin{proof}
$(i)$ Let $\calF=\bigoplus_{i=1}^r\calO(d_i)$. Then 
$(\delta_1, \dots, \delta_r)$ determines a 
homomorphism 
$$
\calF\longrightarrow\calE\ : \ 
(f_1, \dots, f_r)\longmapsto
f_1\delta_1+\dots+f_r\delta_r. 
$$
The Jacobian of this map is an element of 
$\Gamma(\HOM(\calO(d_1+\dots+d_r), \calO(c_1)))=\Gamma(\calO)=\bC$. 
By assumption, the Jacobian is nowhere vanishing, hence 
$\calF\simeq\calE$. 

$(ii)$ Suppose that $\calE|_H=\bigoplus_{i=1}^r F_i$ and 
$F_i\simeq\calO_H(d_i)$. Then by the surjectivity assumption, 
there is $\delta_i\in\Gamma(\calE\otimes\calO(-d_i))$ such that 
$\delta_i|_H$ is a nowhere vanishing section of 
$\Gamma(H, \calF_i\otimes\calO(-d_i))=\bC$. 
Since $\delta_1|_H, \dots, \delta_r|_H$ are linearly independent, 
so are $\delta_1, \dots, \delta_r$. Then by $(i)$, 
$\calE$ is also splitting. 
\end{proof}
\begin{rem}
\normalfont
Comments to those who are already familiar with free arrangements: 
$(i)$ and $(ii)$ in Proposition \ref{prop:easy} 
are analogies of 
Saito's and Ziegler's criteria respectively. 
See Theorem \ref{thm:saitocri} and 
Corollary \ref{cor:zie}. 
\end{rem}
Here we present some criteria for splitting. 

\begin{theorem}
\label{thm:horrocks}
Let $\calF$ be a vector bundle on $\bP^n$. 
\begin{itemize}
\item[$(1)$] (Horrocks) 
Assume that $n\geq 2$. 
$\calF$ is splitting $\Longleftrightarrow$
$H^i(\calF(d))=0$, for any $0<i<n$ and $d\in\bZ$ 
$\Longleftrightarrow$ 
$H^1(\calF(d))=0$, for any $d\in\bZ$. 
\item[$(2)$] (Horrocks) 
Assume that $n\geq 3$. Fix a hyperplane $H\subset\bP^n$. 
Then $\calF$ is splitting $\Longleftrightarrow$ 
$\calF|_H$ is splitting. 
\item[$(3)$] (Elencwajg-Forster, \cite{el-fo}) 
Assume that $n\geq 2$. Let $L\subset\bP^n$ be a generic line and 
set $\calE|_L=\calO_L(d_1)\oplus\dots\oplus\calO_L(d_r)$. Then 
$$
c_2(\calE)\geq\sum_{i<j}d_id_j, 
$$
and $\calE$ is splitting if and only if the equality holds. 
\end{itemize}
\end{theorem}
\begin{proof}
Here we give the proof for $(1)$. The direction $\Longrightarrow$ 
is well-known ($H^i(\bP^n, \calO(d))=0$, for $0<i<n$ and $\forall 
d\in\bZ$). 
Let us assume that $H^1(\bP^n, \calF(d))=0$ for $d\in\bZ$. Let us first 
consider the case $n=2$. By Grothendieck's splitting theorem, 
$\calF|_H$ is splitting. 
By the long exact sequence associated with (\ref{eq:short}), 
we have the surjectivity of 
$\Gamma_*(\calF)\longrightarrow\Gamma_*(\calF|_H)$. 
Hence by Proposition \ref{prop:easy} $(i)$, $\calF$ is splitting. 
For $n\geq3$, it is proved by induction. 
\end{proof}

\begin{rem}
\label{rem:horrocksref}
\normalfont
Horrocks' restriction criterion $(2)$ is generalized to 
reflexive sheaves (\cite{ay-split}). 
Later we will give a refinement of $(3)$ for $n=r=2$ (see 
Theorem \ref{thm:rk2}). 
\end{rem}

\subsection{Basics of arrangements}

Let $V$ be an $\ell$-dimensional vector space. 
A finite set of affine hyperplanes 
$\A=\{H_1, \ldots, H_n\}$ is called a {\it hyperplane arrangement}. 
For each hyperplane $H_i$ we fix a defining 
equation $\alpha_i$ such that $H_i=\alpha_i^{-1}(0)$. 
An arrangement $\A$ is called {\em central} if each $H_i$ passes 
the origin $0\in V$. In this case, the defining equation $\alpha_i
\in V^*$ is linear homogeneous. 
Let $L(\A)$ be the set of non-empty intersections of 
elements of $\A$. Define a partial order on $L(\A)$ 
by $X\leq Y\Longleftrightarrow Y\subseteq X$ for $X, Y\in L(\A)$. 
Note that this is reverse inclusion.

Define a {\it rank function} on $L(\A)$ by $r(X)=\codim X$. 
Denote $L^p(\A)=\{X\in L(\A)|\ r(X)=p\}$. We call $\A$ 
{\it essential} if $L^\ell(\A)\neq\emptyset$. 

Let $\mu:L(\A)\rightarrow \bZ$ be the {\it M\"obius function} 
of $L(\A)$ defined by 
$$
\mu(X)=
\left\{
\begin{array}{ll}
1 &\mbox{ for }X=V\\
-\sum_{Y<X}\mu(Y), &\mbox{ for } X>V. 
\end{array}
\right.
$$
The {\it characteristic polynomial} of $\A$ is 
$\chi(\A, t)=\sum\nolimits_{X\in L(\A)}\mu(X)t^{\dim X}$. 
The characteristic polynomial is characterized by the following 
recursive relations. 
\begin{prop}
Let $\A=\{H_1, \dots, H_n\}$ be a hyperplane arrangement in $V$. 
Let $\A'=\{H_1, \dots, H_{n-1}\}$ and $\A''=H_n\cap\A'$ the induced 
arrangement on $H_n$. Then 
\begin{itemize}
\item in case $\A$ is empty, $\chi(\emptyset, t)=t^{\dim V}$, and 
\item $\chi(\A, t)=\chi(\A', t)-\chi(\A'', t)$. 
\end{itemize}
\end{prop}
We also define the $i$-th Betti number 
$b_i(\A)$ ($i=1, \dots, \ell$) 
by the formula 
$$
\chi(\A, t)
=\sum\nolimits_{i=0}^\ell (-1)^ib_i(\A)t^{\ell-i}. 
$$
This naming and the importance of the characteristic polynomial 
in combinatorics would be justified by the following result. 
\begin{theorem}
\label{thm:chpoly}
$(1)$ If $\A$ is an arrangement in $\bF_q^\ell$ (vector space over 
a finite field $\bF_q$), then 
$|\bF_q^\ell\setminus\bigcup_{H\in\A}H|=\chi(\A, q)$. 

$(2)$ If $\A$ is an arrangement in $\bC^\ell$, then the topological 
$i$-th Betti number of the complement is 
$b_i(\bC^\ell\setminus\bigcup_{H\in\A}H)=b_i(\A)$. 

$(3)$ If $\A$ is an arrangement in $\bR^\ell$, then 
$|\chi(\A, -1)|$ is the number of chambers and 
$|\chi(\A, 1)|$ is the number of bounded chambers. 
\end{theorem}
$(1)$ of the above theorem can be used for the computation 
of $\chi(\A, t)$ for $\A$ defined over $\bQ$. 
It is sometimes called ``Finite field method''. Athanasiadis 
pointed out that we may drop the assumption ``field''. 

\begin{theorem}
Let $\A$ be a hyperplane arrangement such that each $H\in\A$ 
is defined by a linear form $\alpha_H$ of $\bZ$-coefficients. 
For a positive integer $m>0$, consider 
$\overline{H}=\{x\in(\bZ/m\bZ)^\ell\mid \alpha_H(x)\equiv  0\mod m\}$. 
There exists a positive integer $N$ which depends only on $\A$ 
such that if $m>N$ and $m$ is coprime to $N$, then 
$$
|(\bZ/m\bZ)^n\setminus\bigcup_{H\in\A}\overline{H}|=\chi(\A, m). 
$$
\end{theorem}
Athanasiadis systematically used this result to compute 
characteristic polynomials. 
(See \cite{ath-lin, ath-gen}.) 

\subsection{Basics of free arrangements}

Let $V=\bC^\ell$ be a complex vector space with 
coordinate $(x_1, \cdots, x_\ell)$, $\A=\{H_1, \ldots, H_n\}$ 
be a central hyperplane arrangement, namely, 
$0\in H_i$ for all $i=1, \dots, n$. 
We denote by $\Der_V=\bigoplus_{i=1}^\ell S\frac{\partial}{\partial x_i}$ 
the set of polynomial vector fields on $V$ (or $S$-derivations) and 
by 
$\Omega_V^p=\bigoplus_{i_1<\dots <i_p}S 
dx_{i_1}\wedge\dots\wedge dx_{i_p}$ the set of polynomial 
differential $p$-forms. 
\begin{defn}
Let $\theta=\sum_{i=1}^\ell f_i\partial_{x_i}$ be a polynomial vector field. 
$\theta$ is said to be homogeneous of polynomial degree $d$ when 
$f_1, \dots, f_\ell$ are homogeneous polynomial of degree $d$. 
It is denoted by $\pdeg\theta=d$. 
\end{defn}

\begin{rem}
Usually the degree of $\theta$ is considered to be 
$\deg f_i-1$ which is one less than $\pdeg\theta$. 
To avoid confusion, we use the notation $\pdeg\theta$. 
\end{rem}

Let us denote by 
$S=S(V^*)=\bC[x_1, \ldots,x_\ell]$ the polynomial ring and fix 
$\alpha_i\in V^*$ a defining equation of $H_i$, i.e., 
$H_i=\alpha_i^{-1}(0)$.

\begin{defn}
A multiarrangement is a pair $(\A, \m)$ of an arrangement 
$\A$ with a map $ \m: \A\rightarrow \bZ_{\geq 0}$, called 
the multiplicity.  
\end{defn}
An arrangement $\A$ can be identified with a multiarrangement 
with constant multiplicity $m\equiv 1$, which is sometimes called 
a {\em simple arrangement}. 
With this notation, the main object is 
the following module of $S$-derivations which has contact 
to each hyperplane of order $m$. 
We also put $Q=Q(\A,\m)=\prod_{i=1}^n\alpha_i^{\m(H_i)}$ and 
$|\m|=\sum_{H\in\A}\m(H)$. 

\begin{defn}
Let $(\A, \m)$ be a multiarrangement, and define 
the module of vector fields 
logarithmic tangent to $\A$ with multiplicity $\m$ 
(logarithmic vector field) by 
$$
D(\A, \m)=\{\delta\in\Der_V| \delta\alpha_i\in(\alpha_i)^{\m(H_i)},
\forall i\}, 
$$
and differential forms with logarithmic poles along $\A$ 
(logarithmic forms) by 
$$
\Omega^p(\A, \m)=\left\{\left.
\omega\in\frac{1}{Q}\Omega^p_V\right| d\alpha_i\wedge\omega
\mbox{ \normalfont{does not have pole along} }H_i, 
\forall i\right\}. 
$$
\end{defn}
The module $D(\A, \m)$ is obviously a graded $S$-module. 
It is proved in \cite{sai-log} that $D(\A, \m)$ and $\Omega^1(\A, \m)$ 
are dual modules to each other. Therefore, they are reflexive modules. 
A multiarrangement $(\A, \m)$ is said to be {\em free} with 
exponents $(e_1, \ldots, e_\ell)$ if and only if 
$D(\A, \m)$ is an $S$-free module and there exists a 
basis $\delta_1, \ldots,\delta_\ell\in D(\A, \m)$ such that 
$\pdeg \delta_i=e_i$. 
When $\m\equiv 1$, $D(\A, 1)$ and $\Omega^p(\A, 1)$ is denoted 
by $D(\A)$ and $\Omega^p(\A)$ for simplicity. 
An arrangement $\A$ is said to be free if $(\A, 1)$ is free. 
The Euler vector field 
$\theta_E=\sum_{i=1}^\ell x_i\partial_i$ 
is always contained in $D(\A)$ for simple case. 

Let $\delta_1, \ldots, \delta_\ell\in D(\A, \m)$. Then 
$\delta_1\wedge\cdots\wedge\delta_\ell$ is divisible by 
$Q(\A, \m)
\frac{\partial}{\partial x_1}
\wedge\dots\wedge
\frac{\partial}{\partial x_\ell}$. 
The determinant of coefficient matrix of $\delta_1, \dots, 
\delta_\ell$ can be used to characterize freeness. 
\begin{theorem}
\label{thm:saitocri}
(Saito's criterion, \cite{sai-log}) 
Let $\delta_1, \dots, \delta_\ell\in D(\A, \m)$. 
Then the following are equivalent: 
\begin{itemize}
\item[(i)] $D(\A, \m)$ is free with basis $\delta_1, \dots, \delta_\ell$, 
i. e., $D(\A, \m)=S\cdot\delta_1\oplus\dots\oplus S\cdot\delta_\ell$. 
\item[(ii)] 
$\delta_1\wedge\cdots\wedge\delta_\ell=
c\cdot Q(\A, \m)\cdot\frac{\partial}{\partial z_1}
\wedge\cdots\wedge\frac{\partial}{\partial z_\ell}$, where $c\in\bC^*$. 
\item[(iii)] $\delta_1, \dots, \delta_\ell$ are linearly independent 
over $S$ and $\sum_{i-1}^\ell\pdeg\delta_i=|\m|=
\sum_{H\in\A}\m(H)$. 
\end{itemize}
\end{theorem}
From Saito's criterion, we also obtain that if 
a multiarrangement $(\A, \m)$ is free with 
exponents $(e_1, \ldots, e_\ell)$, then 
$|\m|=\sum_{i=1}^\ell e_i$. 

\begin{prop}
If $\A$ is free, then $\A$ is locally free, i.e., 
$\A_X=\{H\in\A\mid X\subset H\}$ is free for any $X\in L(\A)$, $X\neq 0$. 
\end{prop}

For simple arrangement case, there is a good connection of these 
modules with the characteristic polynomial. 
The following result shows that the graded module structure 
of $D(\A)$ determines 
the characteristic polynomial $\chi(\A, t)$. 

\begin{theorem}
\label{thm:stf}
(Solomon-Terao's formula \cite{st-stf}) 
Denote by $\Hilb(\Omega^p(\A), x)\in\bZ[[x]][x^{-1}]$ the 
Hilbert series of the graded module $\Omega^p(\A)$. Define 
\begin{equation}
\label{eq:phi}
\Phi(\A; x, y)=
\sum_{p=0}^\ell \Hilb(\Omega^p(\A), x)y^p. 
\end{equation}
Then 
\begin{equation}
\label{eq:stf}
\chi(\A, t)=\lim_{x\rightarrow 1}
\Phi(\A; x, t(1-x)-1). 
\end{equation}
\end{theorem}

In particular, for free arrangements, we have 
the following beautiful formula, which is 
known as Terao's factorization theorem. 

\begin{theorem}
\label{thm:tfact}
(\cite{ter-fact}) 
Suppose that $\A$ is a free arrangement with 
exponents $(e_1, \dots, e_\ell)$. Then 
\begin{equation}
\label{eq:fact}
\chi(\A, t)=\prod_{i=1}^\ell(t-e_i).
\end{equation}
\end{theorem}

\begin{rem}
There is a notion of the characteristic polynomial 
of a multiarrangement $(\A, \m)$ \cite{atw-char}. 
However it can not be 
defined combinatorially, rather by the Solomon-Terao's formula 
for $\Omega^p(\A, \m)$. 
\end{rem}

\begin{example}
\normalfont
(Braid arrangement or $A_{n-1}$-type arrangement) 
Let $H_{ij}=\{(x_1, \dots, x_\ell)\in\bC^\ell\mid x_i=x_j\}$. 
Consider the arrangement $\A=\{H_{ij}\mid 1\leq i<j\leq n\}$. 
In other words $Q(\A)=\prod_{i<j}(x_i-x_j)$. 

The characteristic polynomial of this arrangement is easily computed 
by the finite field method. For the complement with $\otimes\bF_q$ is 
expressed as 
$$
\{(x_1, \dots, x_n)\in \bF_q^n\mid 
x_i\neq x_j, \mbox{ for } i\neq j\}. 
$$
It is naturally bijective to (ordered) choices of 
$n$ distinct elements from $\bF_q$. 
Hence the cardinality is 
$$
|\bF_q^n\setminus\bigcup_{i<j}H_{ij}|
=
q(q-1)\dots(q-n+1), 
$$
then we have $\chi(\A, t)=t(t-1)(t-2)\dots(t-n+1)$. 

Furthermore, $\A$ is a free arrangement. Indeed set 
\begin{equation*}
\begin{split}
\delta_0&=\partial_{x_1}+\partial_{x_2}+\dots+\partial_{x_n},\\
\delta_1&=x_1\partial_{x_1}+x_2\partial_{x_2}+\dots+x_n\partial_{x_n},\\
\delta_2&=x_1^2\partial_{x_1}+x_2^2\partial_{x_2}+\dots+x_n^2\partial_{x_n},\\
&\dots\\
\delta_{n-1}&=x_1^{n-1}\partial_{x_1}+x_2^{n-1}\partial_{x_2}+
\dots+x_n^{n-1}\partial_{x_n}. 
\end{split}
\end{equation*}
Then $\delta_k(x_i-x_j)=x_i^k-x_j^k$, which is divisible by $(x_i-x_j)$. 
Hence $\delta_k\in D(\A)$. Furthermore, 
by Vandermonde's formula 
$$
\det
\begin{pmatrix}
1&1&\dots&1\\
x_1&x_2&\dots&x_n\\
x_1^2&x_2^2&\dots&x_n^2\\
\vdots&\vdots&\ddots&\vdots\\
x_1^{n-1}&x_2^{n-1}&\dots&x_n^{n-1}
\end{pmatrix}
=
\prod_{1\leq i< j\leq n}
(x_j-x_i), 
$$
and by Saito's criterion, we may conclude that 
$\delta_0, \dots, \delta_{n-1}$ 
is a basis of $D(\A)$. Hence $\A$ is free with exponents 
$(0, 1, \dots, n-1)$. 
\end{example}

To conclude this section, we note that 
the module of logarithmic vector fields 
is recovered from the sheafification: 
\begin{equation}
\label{eq:globalsection}
D(\A, \m)\stackrel{\simeq}{\longrightarrow}
\Gamma_*(\widetilde{D(\A, \m)}). 
\end{equation}
Therefore freeness of $(\A, \m)$ is equivalent 
to the splitting of $D(\A, \m)$. 
\begin{prop}
\label{prop:split}
$(\A, \m)$ is free with exponents $(d_1, \dots, d_\ell)$ 
if and only if 
$\widetilde{D(\A, \m)}\simeq\calO(-d_1)\oplus\calO(-d_2)\oplus
\dots\oplus\calO(-d_\ell)$. 
\end{prop}

\subsection{$2$-multiarrangements}


A simple arrangement $\A=\{H_1, \dots, H_n\}$ 
in dimension two is always free with exponents $(1, n-1)$. 
We can construct an example of basis explicitly as follows: 
$$
\delta_1=x\partial_x+y\partial_y,\ 
\delta_2=(\partial_y Q)\partial_x-(\partial_x Q)\partial_y. 
$$
The multiarrangement $(\A, \m)$ in dimension two is also always free. 
There are two ways to prove it. First idea is based on 
$D(\A, \m)$ being a reflexive module. Then $2$-dimensional and reflexivity 
implies freeness. Another idea is based on the isomorphism 
$$
D(\A, \m)\stackrel{\simeq}{\longrightarrow}
\Gamma_*(\widetilde{D(\A, \m)}). 
$$
If $\A$ is in dimension two, the sheafification 
$\widetilde{D(\A, \m)}$ is a 
torsion free sheaf on $\bP^1$. By Grothendieck splitting theorem, 
we conclude $D(\A, \m)$ is free. We have the following. 
\begin{prop}
Let $(\A, \m)$ be a $2$-multiarrangement. Then 
it is free and the exponents $(d_1, d_2)$ satisfy 
$d_1+d_2=|\m|$. 
\end{prop}
The determination of exponents of $2$-multiarrangements is difficult, but 
it is an important problem because it is related to the freeness of 
$3$-arrangements (see \S\ref{subsec:multi}). 
The following lemma is useful for the computation of the exponents. 
\begin{lem}
\label{lem:half}
Let $(\A, \m)$ be a $2$-multiarrangement. Let $\delta\in D(\A, \m)$. 
Assume that 
$d=\pdeg\delta\leq\frac{|\m|}{2}$ and no nontrivial divisor of $\delta$ is 
contained in $D(\A, \m)$. Then $\exp(\A, \m)=(d, |\m|-d)$. 
\end{lem}
\begin{proof}
Suppose that $\exp(\A, \m)=(d_1, d_2)$ with $d_1\leq d_2$. Then 
clearly $d_1\leq d$. 
There exists $\delta_1$ of $\pdeg\delta_1=d_1$. Since 
$d\leq\frac{|\m|}{2}=\frac{d_1+d_2}{2}$, we have $d_1\leq d\leq d_2$. 
If $d_1<d$, then we have $d_1<d<d_2$. 
Hence $\delta$ can be expressed as $\delta=F\cdot\delta_1$ with some 
polynomial $F$ of $\deg F>0$. 
But this contradicts the assumption that no nontrivial divisor of 
$\delta$ is contained in $D(\A, \m)$. So $\deg F=0$ and we have 
$d_1=d, d_2=|\m|-d$. 
\end{proof}

For the following cases we can determine the exponents 
combinatorially.

\begin{prop}
\label{prop:typical}
\normalfont
Let $(\A, \m)$ be a $2$-multiarrangement. We may assume that 
$m_i=\m(H_i)$ satisfies $m_1\geq m_2\geq\dots\geq m_n>0$. Set 
$m=\sum_{i=1}^n m_i$. 
\begin{itemize}
\item[(i)] If $m_1\geq\frac{m}{2}$, then 
the exponents are $\exp(\A, \m)=(m_1, m-m_1)$. 
\item[(ii)]
if $n\geq\frac{m}{2}+1$, then $\exp(\A, \m)=
(m-n+1, n-1)$. 
\item[(iii)] 
If $m_1=m_2=\dots=m_n=2$, then $\exp(\A, \m)=(n, n)$. 
\item[(iv)] 
If $n=3$ and $m_1\leq m_2+m_3$, then 
$$
\exp(\A, \m)=
\left\{
\begin{array}{ll}
(k,k), &\mbox{ if }|\m|=2k,\\
(k,k+1), &\mbox{ if }|\m|=2k+1.
\end{array}
\right.
$$
\end{itemize}
\end{prop}
\begin{proof}
(i) We can set coordinates $(x,y)$ such that 
$H_1=\{x=0\}$, in other words, $\alpha_1=x$. 
Set $\delta=(\prod_{i=2}^n\alpha^{m_i})\cdot\partial_y$. 
Then $\delta x=0$ and $\delta\alpha_i\in(\alpha_i)^{m_i}$ for 
$i\geq 2$. Hence $\delta\in D(\A, \m)$. We also have 
$$
\pdeg \delta=m_2+\dots m_n=|\m|-m_1\leq\frac{|\m|}{2}, 
$$
and no divisor of $\delta$ is not contained in 
$D(\A, \m)$. 
From Lemma \ref{lem:half}, $\exp(\A, \m)=(m_1, |\m|-m_1)$. 

(ii) Let us define $\delta$ as 
$$
\delta=
\frac{\prod_{i=1}^n\alpha_i^{m_i}}{\prod_{i=1}^n\alpha_i}
\cdot\theta_E, 
$$
where $\theta_E=x\partial_x+y\partial_y$ is the Euler vector field. Then 
since $\theta_E\alpha=\alpha$ for any linear form $\alpha$, 
$\delta\in D(\A, \m)$. From the assumption, we have 
$$
\pdeg\delta=|\m|-n+1\leq n-1. 
$$
Since $(|\m|-n+1)+(n-1)=|\m|$, we have 
$|\m|-n+1=\pdeg\delta\leq \frac{|\m|}{2}$. It is also easily 
checked that $\delta$ does not have non trivial divisor which 
is contained in $D(\A, \m)$. Hence we have 
$\exp(\A, \m)=(|\m|-n+1, n-1)$. 

(iii) and (iv) are proved by 
explicit constructions of basis. See 
\cite[Exapmple 2.2]{wake-yuz} and \cite{waka-exp} 
respectively. (Both are highly nontrivial.) In \S\ref{subsec:conn} 
we present an alternative proof of (iii) given by T. Abe. 
\end{proof}
Thus if either $\max\{\m(H)\mid H\in\A\}$ is large 
(not less than the half of $|\m|=\sum\m(H)$)
 or the number of lines $n=|\A|$ is large 
(not less than $\frac{|\m|}{2}+1$), then the exponents are 
combinatorially determined. 
This motivates us to give the following definition. 
\begin{defn}
The multiplicity $\m:\A\rightarrow\bZ_{\geq 0}$ is said to be 
{\em balanced} if $\m(H)\leq\frac{\sum_{H\in\A}\m(H)}{2}$ for 
all $H\in\A$. 
\end{defn}
As we have seen, if the multiplicity is not balanced, then the exponents are 
combinatorially determined. 
However the exponents are 
not combinatorially determined for balanced cases in general. 

\begin{example}
Let $(\A_t, \m)$ be a multiarrangement defined by 
$$
Q(\A_t, \m)=x^3y^3(x+y)^1(tx-y)^1, 
$$
where $t\in\bC\setminus\{0,-1\}$. Then exponents 
are 
$$
\exp(\A_t, \m)=
\left\{
\begin{array}{cl}
(3,5), &\mbox{ if }t=1, \\
(4,4), &\mbox{ if }t\neq1.
\end{array}
\right.
$$
Indeed, it is easily seen that 
\begin{equation*}
\begin{split}
\delta_1&=x^3\partial_x+y^3\partial_y, \\
\delta_2&=x^5\partial_x+y^5\partial_y, 
\end{split}
\end{equation*}
form a basis of $D(\A_{1},\m)$. For $t\neq 1$ (and $t\neq 0, -1$), 
$\delta_1\notin D(\A_{t},\m)$. But $(tx-y)\delta_1\in D(\A_{t}, \m)$ with 
$\pdeg=4$. If there exists an element of $D(\A_{t}, \m)$ of $\pdeg=3$, 
it should be a divisor of $(tx-y)\delta_1$. It is impossible. 
Thus exponents for other cases are $(4,4)$. 
\end{example}
We may observe that any $4$-lines can be moved 
by $PGL_2(\bC)$-action to 
$xy(x+y)(tx-y)$ with $t\in\bC\setminus\{0, -1\}$. 
On a Zariski open subset of the parameter space $\bC\setminus\{0, 1, -1\}
\subset\bC\setminus\{0, -1\}$, 
the exponents are $(4,4)$ and at $t=1$, they become $(3,5)$. 
This generally happens. 
We shall prove the upper-semicontinuity on the parameter space 
of the following function. 

\begin{defn}
Put $\exp(\A, \m)=(d_1, d_2)$. Then we denote 
$$
\Delta(\A, \m)=|d_1-d_2|. 
$$
\end{defn}
The difference of exponents $\Delta(\A, \m)$ is a function on $\A$ 
and $\m$. We first fix the multiplicity $\m$. 
The parameter space of $\A$ can be described as 
$$
\calM_n=\{(H_1, \dots, H_n)\in(\bP^{1*})^n\mid 
H_i\neq H_j,\ \mbox{ for }i\neq j\}
$$
\begin{prop}
\label{prop:upper}
Fix the multiplicity $\m:\{1, \dots, n\}\rightarrow\bZ_{>0}$. Then 
$$
\Delta:\calM_n\longrightarrow \bZ_{>0},\ (\A\longmapsto 
\Delta(\A, \m))
$$
is upper semi-continuous, i. e., the subset 
$\{\Delta<k\}\subset\calM_n$ is 
a Zariski open subset for any $k\in\bR$. 
\end{prop}

\begin{proof}
It suffices to prove that $\{\Delta\geq k\}$ is Zariski closed in 
$\calM_n$. Since $d_1+d_2=|\m|$, $\Delta(\A, \m)\geq k$ if and only if 
there exists $\delta\in D(\A, \m)$ such that 
$\pdeg\delta\leq\lfloor\frac{|\m|}{2}-\frac{k}{2}\rfloor$. Thus we consider 
when $\delta\in D(\A, \m)$ of 
$\pdeg\delta=\lfloor\frac{|\m|}{2}-\frac{k}{2}\rfloor$ exists. 
Put $d=\lfloor\frac{|\m|}{2}-\frac{k}{2}\rfloor$, 
$\alpha_i=p_ix+q_iy$ and 
$$
\delta=
(a_0x^d+a_1x^{d-1}y+\dots+a_dy^d)\partial_x+
(b_0x^d+b_1x^{d-1}y+\dots+b_dy^d)\partial_y. 
$$
The assertion $\delta\alpha_i\in(\alpha_i^{m_i})$ is equivalent to 
\begin{eqnarray}
\label{eq:lineq}
\delta\alpha_i=(p_ix+q_iy)^{m_i}
(c_0x^{d-m_i}+c_1x^{d-m_i-1}y+\dots c_{d-m_i}y^{d-m_i}), 
\end{eqnarray}
for some $c_0, c_1, \dots$. 
Hence the existence of $\delta\in D(\A, \m)$ of degree $d$ is 
equivalent to the existence of the solution to the 
system (\ref{eq:lineq}) of linear equations on $a_i, b_i$ and $c_i$. 
It is a Zariski closed condition on the parameters $p_i$ and $q_i$.  
\end{proof}
The following are the two fundamental results on exponents of 
$2$-multiarrangements. 

\begin{theorem}
\label{thm:2exp}
Let $\m:\{1, \dots, n\}\longrightarrow\bZ_{>0}$ be a balanced multiplicity and 
$\A=\{H_1, \dots, H_n\}$ a $2$-arrangement. 
\begin{itemize}
\item[(i)] (Wakefield-Yuzvinsky \cite{wake-yuz}) 
For generic $\A$, $\Delta(\A, \m)\leq 1$. 
\item[(ii)] (Abe \cite{abe-exp}) 
$\Delta(\A, \m)\leq n-2$. 
\end{itemize}
\end{theorem}
The proof of (i) is a careful extension of that of 
upper semicontinuity (Proposition \ref{prop:upper}). 
See cited papers for proof. 
The proof of (ii) is of a very different nature. Abe 
(\cite{abe-exp} and Abe-Numata \cite{an-lat}) first fix 
$\A$ and then consider $\Delta$ as a function from 
the set of multiplicities $\bZ_{>0}^n$ to $\bZ_{\geq 0}$, 
$$
\Delta:\bZ_{>0}^n\longrightarrow\bZ_{\geq 0} ,\ \m\longmapsto 
\Delta(\A, \m). 
$$
They studied the structure of this function in great detail. 
The proof of (ii) is based on this. 

(i) tells the generic behavior of the function $\Delta$. (ii) tells 
the upper bound of $\Delta$ for balanced multiplicities. 
As far as the author knows, the examples $(\A, \m)$ attaining the upper bound 
of $\Delta$ are related to 
interesting free arrangements of rank $3$. 
Abe found a class of free arrangements which are combinatorially 
characterizable \cite{abe-exp}. 
See Example \ref{ex:g2cat}. 

\begin{pbm}
Give a unified proof for Theorem \ref{thm:2exp} (i) and (ii). 
\end{pbm}

\subsection{Multiarrangements and free arrangements}
\label{subsec:multi}

Multiarrangements appear as restrictions of simple arrangements. 
Namely, let $\A$ be an arrangement in $V$ of $\dim V=\ell$. 
For $H\in\A$ let us denote by $\A^H$ the induced arrangement on $H$. 

\begin{defn}
Define the function $\m^H:\A^H\longrightarrow\bZ_{>0}$ by 
$$
X\in\A^H\longmapsto \sharp\{H'\in\A\mid H'\supset X\}-1. 
$$
We call $(\A^H, \m^H)$ the Ziegler's multirestriction. 
\end{defn}

\begin{example}
Let $V=\bC^3$ with coordinates $x, y, z$. Put 
$
H_1=\{z=0\}, 
H_2=\{x=0\}, 
H_3=\{y=0\}, 
H_4=\{x-z=0\}, 
H_5=\{x+z=0\}, 
H_6=\{y-z=0\}, 
H_7=\{y+z=0\}, 
H_8=\{x-y=0\}, 
H_9=\{x+y=0\}$. Then $\A=\{H_1, \dots, H_9\}$ is free 
with exponents $(1, 3, 5)$. 
Ziegler's multirestriction to 
$(\A^{H_1}, \m^{H_1})$ is $x^3y^3(x-y)(x+y)$. 
(See Figure \ref{fig:ex}) 

\begin{figure}[htbp]
\begin{picture}(100,150)(130,0)
\thicklines

\put(0,150){\line(1,0){140}}
\qbezier(140,150)(150,150)(150,140)
\put(150,0){\line(0,1){140}}
\put(153,5){$H_1$}

\multiput(50,0)(20,0){3}{\line(0,1){130}}
\put(70,150){\line(0,-1){20}}
\put(70,150){\line(1,-1){20}}
\put(70,150){\line(-1,-1){20}}

\multiput(0,50)(0,20){3}{\line(1,0){130}}
\put(150,70){\line(-1,0){20}}
\put(150,70){\line(-1,1){20}}
\put(150,70){\line(-1,-1){20}}

\put(10,10){\line(1,1){138}}
\put(10,130){\line(1,-1){130}}

\put(200,70){\line(1,0){150}}
\put(270,00){\line(0,1){150}}
\put(200,0){\line(1,1){130}}
\put(200,140){\line(1,-1){130}}

\put(205,135){$1$}
\put(275,135){$3$}
\put(330,120){$1$}
\put(330,73){$3$}

\end{picture}
\caption{$\A=\{H_1, \dots, H_9\}$ and $(\A^{H_1}, \m^{H_1})$}
\label{fig:ex}
\end{figure}
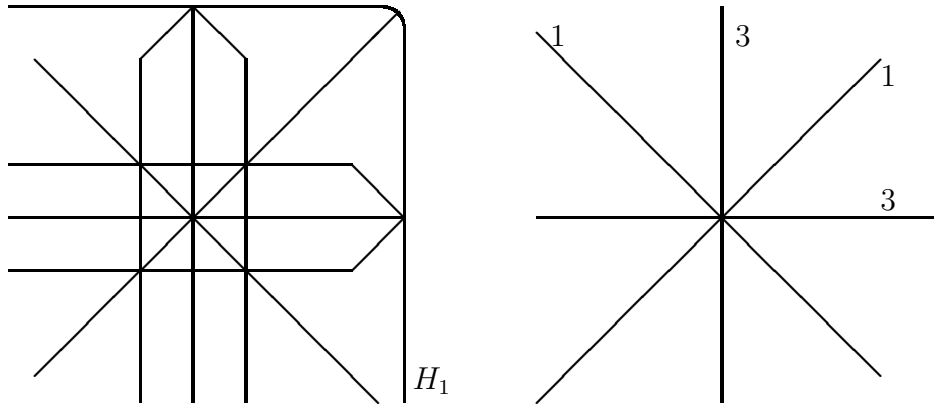

\end{example}

\begin{defn}
Fix a hyperplane $H_1\in\A$. Then we define a submodule $D_1(\A)$ of 
$D(\A)$ by 
$$
D_1(\A)=
\{\delta\in D(\A)\mid \delta\alpha_{H_1}=0\}. 
$$
\end{defn}

\begin{lem}
Under the above notations, 
$D(\A)=S\cdot\theta_E\oplus D_1(\A)$. 
\end{lem}
\begin{proof}
Let $\delta\in D(\A)$. Since 
$\delta-\frac{\delta\alpha_{H_1}}{\alpha_{H_1}}\cdot\theta_E$ is 
in $D_1(\A)$, $\delta=
\frac{\delta\alpha_{H_1}}{\alpha_{H_1}}\cdot\theta_E+
\left(\delta-\frac{\delta\alpha_{H_1}}{\alpha_{H_1}}\cdot\theta_E\right)$ 
gives the desired decomposition. 
\end{proof}

\begin{theorem} 
(Ziegler \cite{zie-multi}) 
Notations as above. 
\begin{itemize}
\item[(i)] If $\delta\in D_1(\A)$, then $\delta|_{H_1}\in 
D(\A^{H_1}, \m^{H_1})$. 
\item[(ii)] 
If $\A$ is free with exponents $(1, d_2, \dots, d_\ell)$, 
then $(\A^{H_1}, \m^{H_1})$ is free with exponents 
$(d_2, \dots, d_\ell)$. 
\end{itemize}
\end{theorem}
\begin{proof}
We can choose coordinates $x_1, \dots, x_\ell$ in such a way that 
$x_1=\alpha_{H_1}$. Let $X\in\A^{H_1}$ and put 
$$
\A_X=\{H\in\A\mid H\supset X\}=
\{H_1, H_{i_1}, H_{i_2}, \dots, H_{i_m}\}. 
$$
Since $H\cap H_{i_1}=\dots=H\cap H_{i_m}=X$, the restriction 
$\alpha_{i_p}|_{x_1=0}$ determines the same hyperplane. 
Thus we may assume that $\alpha_{i_p}$ have the form 
\begin{equation*}
\begin{split}
\alpha_{i_1}(x_1, \dots x_\ell)&=
c_1x_1+\alpha'(x_2, \dots, x_\ell)\\
\alpha_{i_2}(x_1, \dots x_\ell)&=
c_2x_1+\alpha'(x_2, \dots, x_\ell)\\
\dots&\dots\\
\alpha_{i_m}(x_1, \dots x_\ell)&=
c_mx_1+\alpha'(x_2, \dots, x_\ell),  
\end{split}
\end{equation*}
where $c_1, \dots, c_m$ are mutually distinct. 
Let $\delta\in D_1(\A)$. By definition, 
$$
\delta(c_kx_1+\alpha'(x_2, \dots, x_\ell))\in
(c_kx_1+\alpha'(x_2, \dots, x_\ell)). 
$$
Then since $\delta x_1=0$, 
$\delta\alpha'(x_2, \dots, x_\ell)$ is divisible by 
$c_kx_1+\alpha'(x_2, \dots, x_\ell)$ for all $k=1, \dots, m$. 
Hence it is divisible by 
$$
\prod_{k=1}^m
(c_kx_1+\alpha'(x_2, \dots, x_\ell)). 
$$
Now we restrict to $x_1=0$. Then $\delta|_{x_1=0}\alpha'$ 
is divisible by $(\alpha')^m$. Thus (i) is proved. 

(ii) Let $\delta_1=\theta_E, \delta_2, \dots, \delta_\ell$ be a 
basis of $D(\A)$ such that $\delta_2, \dots, \delta_\ell
\in D_1(\A)$. Let us set 
$\delta_i=\sum_{j=2}^\ell f_{ij}\partial_{x_i}$. 
We will prove that $\delta_2|_{x_1=0}, \dots, \delta_\ell|_{x_1=0}$ 
are linearly independent over $S/x_1S=\bC[x_2, \dots, x_\ell]$. 
Indeed by Saito's criterion, the determinant 
$$
\det
\begin{pmatrix}
x_1&x_2&\dots&x_\ell\\
0&f_{22}&\dots&f_{2n}\\
\vdots&\vdots&\ddots&\vdots\\
0&f_{n2}&\dots&f_{nn}
\end{pmatrix}
=
x_1\cdot
\det
\begin{pmatrix}
f_{22}&\dots&f_{2n}\\
\vdots&\ddots&\vdots\\
f_{n2}&\dots&f_{nn}
\end{pmatrix}
$$
is divisible by $x_1$ just once. 
Hence $\det(f_{ij})$ is not divisible by $x_1$, which implies that 
$\delta_2|_{x_1=0}, \dots, \delta_\ell|_{x_1=0}$ 
are linearly independent over $S/x_1S=\bC[x_2, \dots, x_\ell]$. 
Furthermore, we have 
$$
\sum_{i=2}^\ell
\pdeg\delta_i|_{x_1=0}
=|\A|-1
=\sum_{X\in\A^{H_1}}\m^{H_1}(X). 
$$
By Saito's criterion, they form a free basis of 
$D(\A^{H_1}, \m^{H_1})$. 
\end{proof}
It seems natural to pay attention to the exact sequence 
\begin{equation}
\label{eq:rest}
0\longrightarrow
D_1(\A)\stackrel{x_1\cdot}{\longrightarrow}
D_1(\A)\stackrel{\rho}{\longrightarrow}
D(\A^{H_1}, \m^{H_1}). 
\end{equation}
From the above proof, we know that if $\A$ is free, 
then the restriction map $\rho$ is surjective. 

\begin{corollary}
\label{cor:zie}
If the restriction map $\rho$ is surjective, and 
$D(\A^{H_1}, \m^{H_1})$ is free with exponents 
$(d_2, \dots, d_\ell)$, then $\A$ is free with exponents 
$(1, d_2, \dots, d_\ell)$. 
\end{corollary}
\begin{proof}
By the assumption, there exists $\delta_2, \dots, \delta_\ell
\in D_1(\A)$ such that $\rho(\delta_2)=\delta_2|_{x_1=0}, \dots, 
\rho(\delta_\ell)=\delta_\ell|_{x_1=0}$ are basis of 
$D(\A^{H_1}, \m^{H_1})$. As in the previous proof, 
$\delta_2, \dots, \delta_\ell$ and $\theta_E$ are 
linearly independent and the sum of $\pdeg$ is 
$|\A|$. Hence by Saito's criterion, $(\theta_E, \delta_2, \dots, 
\delta_\ell)$ is a basis of $D(\A)$. 
\end{proof}

Generally, $\rho$ is not surjective. However, local freeness 
implies local surjectivity. 

\begin{defn}
Let $\A$ be an arrangement and $H_1\in\A$. Then 
$\A$ is said to be locally free along $H_1$ if 
$\A_X=\{H\in\A\mid X\subset H\}$ is free for all $X\in L(\A)$ with 
$X\subset H_1$ and $X\neq 0$. 
\end{defn}
Local freeness along $H_1$ implies 
$$
0\longrightarrow
D_1(\A_X)\stackrel{x_1\cdot}{\longrightarrow}
D_1(\A_X)\stackrel{\rho}{\longrightarrow}
D(\A_X^{H_1}, \m_X^{H_1})
\longrightarrow 0
$$
for all $X\in L(\A), X\neq 0$ with $X\subset H_1$. Thus we have 
an exact sequence of sheaves over $\bP^{\ell-1}$. 
\begin{equation}
\label{eq:rest}
0\longrightarrow
\widetilde{D_1(\A)}(-1)\stackrel{x_1\cdot}{\longrightarrow}
\widetilde{D_1(\A)}\stackrel{\rho}{\longrightarrow}
\widetilde{D(\A^{H_1}, \m^{H_1})}
\longrightarrow 0.  
\end{equation}
Thus we obtain a relation between Ziegler's multirestriction 
and restriction of the sheaf $D_1(\A)$. 
\begin{prop}
\label{prop:stalk}
If $\A$ is locally free along $H_1$, then 
$$
\widetilde{D_1(\A)}|_{H_1}=
\widetilde{D(\A^{H_1}, \m^{H_1})}. 
$$
\end{prop}
By the above proposition combined with 
Proposition \ref{prop:split} and 
Horrocks criterion (Theorem \ref{thm:horrocks} (2), see 
also subsequent Remark \ref{rem:horrocksref}), 
we have the following criterion for freeness. 
\begin{theorem}
\label{thm:char4}
(\cite{yos-char}) 
Assume that $\ell\geq 4$. Then $\A$ 
is free with exponents $(1, d_2, \dots, d_\ell)$ if and only if 
the following conditions are satisfied. 
\begin{itemize}
\item $\A$ is locally free along $H_1$, 
\item Ziegler's multirestriction 
$(\A^{H_1}, m^{H_1})$ is free with exponents 
$(d_2, \dots, d_\ell)$. 
\end{itemize}
\end{theorem}

The above criterion is not valid for $\ell=3$. Indeed 
for $\ell=3$, both conditions are automatically satisfied, however, 
there exist non free $3$-arrangements. For characterizing 
freeness of $3$-arrangements, we need characteristic polynomials. 

\begin{theorem}
\label{thm:char3}
(\cite{yos-3arr}) 
Let $\A$ be a $3$-arrangement. Set $\chi(\A, t)=(t-1)(t^2-b_1t+b_2)$ and 
$\exp(\A^{H_1}, \m^{H_1})=(d_1, d_2)$. Then 
\begin{itemize}
\item[(i)] 
$b_2\geq d_1d_2$, furthermore 
$b_2-d_1d_2=\dim\Coker(\rho:D_1(\A)\rightarrow D(\A^{H_1}, \m^{H_1}))$. 
\item[(ii)] 
If $b_2=d_1d_2$, then $\A$ is free with exponents 
$(1, d_1, d_2)$.  
\end{itemize}
\end{theorem}
The proof is based on an analysis of Solomon-Terao's formula. 
Theorem \ref{thm:char3} is also a corollary of a result in 
the next section (Theorem \ref{thm:rk2}). 

By combining Theorem \ref{thm:char4} and \ref{thm:char3}, 
we recently obtained the following criterion for $\ell\geq 4$. 

\begin{theorem}
\label{thm:b2}
(Abe-Yoshinaga \cite{ay-b2}) 
Assume that $\ell\geq 4$ and the multirestriction is free 
with $\exp(\A^{H_1}, \m^{H_1})=(d_2, \dots, d_\ell)$. 
Put 
$\chi(\A, t)=(t-1)(t^{\ell-1}-b_1t^{\ell-2}+b_2t^{\ell-3}-\dots)$. 
Then 
$$
b_2\geq\sum_{2\leq i<j\leq\ell}d_id_j, 
$$
and $\A$ is free if and only if the equality holds. 
\end{theorem}
\begin{rem}
At a glance, this result is similar to that of 
Elencwajg-Forster (Theorem \ref{thm:horrocks} (3) and 
see Bertone-Roggero \cite{br-spl} for torsion free case). 
However at this moment, we can not find any (simple) 
logical implications. 
\end{rem}

\begin{example}
\label{ex:g2cat}
Let $\A=\{H_0, H_1, \dots, H_{18}\}$ 
be the cone of the $G_2$-Catalan arrangement 
$G_2^{[-1,1]}$ (see Figure \ref{fig:G2cat}), where $H_0$ is 
corresponding to the line at infinity. 
Using Abe's inequality 
(Theorem \ref{thm:2exp} (ii)) and Theorem 
\ref{thm:char3}, we can prove the freeness of $\A$ as follows. 
First the characteristic polynomial is 
$$
\chi(\A, t)=(t-1)(t-7)(t-11). 
$$
Let us consider the multirestriction $(\A^{H_0}, \m^{H_0})$. 
Put the exponents $\exp(\A^{H_0}, \m^{H_0})=(d_1, d_2)$. 
Then by Theorem \ref{thm:char3}, 
$$
d_1d_2\leq 77.
$$
Since the multirestriction is balanced, by Abe's inequality, 
we have 
$$
|d_1-d_2|\leq 6-2=4.
$$
Combining these two inequalities, we have $d_1d_2=77$ hence 
$\A$ is free with exponents $(1, 7, 11)$. 
\begin{figure}[htbp]
\begin{picture}(400,200)(0,0)
\thicklines
\multiput(0,85.56582)(0,14.43418){3}{\line(1,0){200}}
\multiput(75,0)(25,0){3}{\line(0,1){200}}
\qbezier(200,128.86)(200,128.86)(0,13.4)
\qbezier(200,157.73)(200,157.73)(0,42.27)
\qbezier(200,186.6)(200,186.6)(0,71.14)
\qbezier(0,128.86)(0,128.86)(200,13.4)
\qbezier(0,157.73)(0,157.73)(200,42.27)
\qbezier(0,186.6)(0,186.6)(200,71.14)

\qbezier(141.068,200)(141.068,200)(25.6,0)
\qbezier(157.734,200)(100,100)(42.266,0)
\qbezier(174.401,200)(174.401,200)(58.932,0)
\qbezier(141.068,0)(141.068,0)(25.6,200)
\qbezier(157.734,0)(100,100)(42.266,200)
\qbezier(174.401,0)(174.401,0)(58.932,200)

\put(230,100){\line(1,0){140}}
\put(300,30){\line(0,1){140}}
\qbezier(335,160.62)(300,100)(265,39.38)
\qbezier(265,160.62)(300,100)(335,39.38)
\qbezier(360.62,135)(300,100)(239.38,65)
\qbezier(360.62,65)(300,100)(239.38,135)

\put(230,135){$3$}
\put(363,135){$3$}
\put(260,163){$3$}
\put(333,163){$3$}
\put(297,173){$3$}
\put(373,97){$3$}



\end{picture}
      \caption{$G_2^{[-1,1]}$ and restriction of its cone 
$cG_2^{[-1,1]}$ to $H_0$.}
\label{fig:G2cat}
\end{figure}
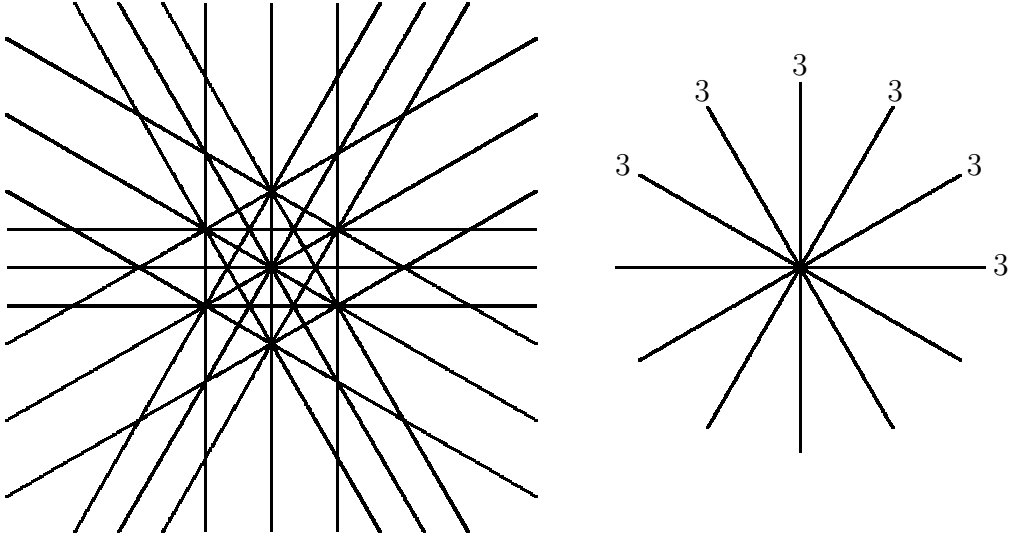
\end{example}
We emphasise that in the above example, 
only the computation of characteristic 
polynomial is enough to prove freeness. 

\subsection{Characteristic polynomials and Chern polynomials}

Let $\A$ be an arrangement in $V$ of $\dim V=\ell$. 
By Terao's factorization theorem, if $\A$ is free with exponents 
$(d_1, \dots, d_\ell)$, then 
$$
\chi(\A, t)=\prod_{i=1}^\ell(t-d_i). 
$$
On the other hand, the sheafification splits 
$\widetilde{D(\A)}=\calO_{\bP^{\ell-1}}(-d_1)\oplus\dots\oplus
\calO_{\bP^{\ell-1}}(-d_1)$. The Chern polynomial of this 
sheaf is 
\begin{equation}
\begin{split}
c_t(\widetilde{D(\A)})
&=
\sum_{i=1}^{\ell-1}c_i(\widetilde{D(\A)})t^i\\
&\equiv
\prod_{i=1}^\ell(1-d_it) \mod t^\ell,  
\end{split}
\end{equation}
where $c_i(-)$ is $i$-th Chern number. It is easily seen that 
these two polynomials are related by the following formula 
\begin{equation}
\label{eq:ch}
t^\ell\cdot
\chi(\A, \frac{1}{t})=
c_t(\widetilde{D(\A)})
\mod t^\ell. 
\end{equation}
Note that the left hand side of (\ref{eq:ch}) is computed by 
Solomon-Terao's formula (Theorem \ref{thm:stf}). 
Musta\c t\v a and Schenck proved that a similar formula computes 
the Chern polynomial for arbitrary vector bundle on the 
projective space. 

\begin{theorem}
\label{thm:ms}
(Musta\c t\v a and Schenck \cite{mus-sch}) 
Let $\calE$ be a vector bundle over $\bP^n$ of rank $r$. Then 
$$
c_t(\calE)=
\lim_{x\rightarrow 1}
(-t)^r(1-x)^{n+1-r}
\sum_{i=0}^r \Hilb(\Gamma_*(\bigwedge^i\calE), x)
\left(
\frac{x-1}{t}-1
\right)^i. 
$$
\end{theorem}
As a corollary, we have: 
\begin{corollary}
Let $\A$ be a locally free arrangement. Then 
the formula (\ref{eq:ch}) holds. 
\end{corollary}

Using Musta\c t\v a-Schenck, we can prove the following. 
\begin{theorem}
\label{thm:rk2}
Let $\calE$ be a rank two vector bundle on 
$\bP^2$. Let $L\subset \bP^2$ be a line. Put 
$\calE|_L=\calO_L(d_1)\oplus\calO_L(d_2)$. Then 
$c_2(\calE)\geq d_1d_2$, furthermore 
$$
c_2(\calE)-d_1d_2=
\dim\Coker(\Gamma_*(\calE)\longrightarrow
\Gamma_*(\calE|_L)). 
$$
$\calE$ is splitting if and only if $c_2(\calE)=d_1d_2$. 
\end{theorem}
\begin{proof}
By Theorem \ref{thm:ms}, the second Chern class is 
$$
c_2(\calE)=\lim_{x\rightarrow 1}
\left(
\frac{1}{(1-x)^2}
-(1-x)\Hilb(\Gamma_*(\calE),x)+
\frac{x^{-c_1(\calE)}}{(1-x)^2}
\right). 
$$
On the other hand, $c_1(\calE)=d_1+d_2$ and 
$$
d_1d_2=
\lim_{x\rightarrow 1}
\left(
\frac{1}{(1-x)^2}-
\frac{x^{-d_1}+x^{-d_2}}{(1-x)^2}+
\frac{x^{-d_1-d_2}}{(1-x)^2}
\right). 
$$
Hence 
\begin{equation*}
\begin{split}
c_2(\calE)-d_1d_2&=
\lim_{x\rightarrow 1}
\left(
\frac{x^{-d_1}+x^{-d_2}}{(1-x)^2}
-(1-x)\Hilb(\Gamma_*(\calE),x)
\right)\\
&=
\lim_{x\rightarrow 1}
\left(
\Hilb(\Gamma_*(\calE|_L), x)-
\Hilb(\Image(\Gamma_*(\calE)\rightarrow\Gamma_*(\calE|_L)), x)
\right)\\
&=
\dim\Coker(\Gamma_*(\calE)\longrightarrow
\Gamma_*(\calE|_L)). 
\end{split}
\end{equation*}

\end{proof}

\subsection{Around Terao Conjecture}

In \cite{ter-conj}, Terao posed the following problem. 
\begin{pbm}
Let $\A_1, \A_2$ be arrangements in $V$ s. t. 
$L(\A_1)\simeq L(\A_2)$. Assume that $\A_1$ is free. Then 
is $\A_2$ also free? 
\end{pbm}
It is obviously true in dimension 
$2$. However the cases $\ell\geq 3$ 
are still open. In view of Theorem \ref{thm:char3}, 
if the exponents of multirestriction were determined 
combinatorially, the freeness is also determined combinatorially. 
\begin{prop}
Let $\A_1, \A_2$ be in $V$ of $\dim V=\ell=3$ such that 
$L(\A_1)\simeq L(\A_2)$. Assume that $\A_1$ is free. 
If there exists a hyperplane $H\in\A$ such that 
the multirestriction $(\A^H, \m^H)$ satisfies 
one of conditions in Proposition \ref{prop:typical}, 
then $\A_2$ is also free. 
\end{prop}
Thus the difficulty of Terao's conjecture for $\ell=3$ is 
equivalent to the difficulty of determining 
exponents of $2$-multiarrangements. 

A possible approach to Terao's conjecture is to look 
at the set of arrangements which have prescribed 
intersection lattice, and then analyze the freeness on the set. 
We first introduce such a set, the parameter space of 
arrangements having the fixed lattice. 
Let $\ell\geq 3$, $n\geq 1$. Fix a poset $L$. Then define 
the set $\calM_{\ell,n}(L)$ of arrangements with lattice $L$ by 
$$
\calM_{\ell,n}(L)
=
\{\A=(H_1, \dots, H_n)\in
(\bP^{\ell-1*})^n\mid 
H_i\neq H_j, L(\A)\simeq L\}. 
$$
Terao's conjecture is equivalent to the preservation of the 
freeness/nonfreeness on $\calM_{\ell,n}(L)$. 
Yuzvinsky proved that free arrangements form a Zariski open subset 
in $\calM_{\ell,n}(L)$. 
\begin{theorem}
(Yuzvinsky \cite{yuz-sheaf, yuz-obst, yuz-loc}) 
$$
\calM_{\ell,n}(L)^{\tiny \mbox{free}}=
\{\A\in\calM_{\ell,n}(L)\mid 
\mbox{$\A$ is free }
\}
$$
is a Zariski open subset of $\calM_{\ell,n}(L)$. 
\end{theorem}
In his proof, Yuzvinsky defines lattice cohomology using the 
structure of $L(\A)$ and $D(\A)$. Then he characterizes 
the freeness of $\A$ via vanishings of these cohomology groups. 
The statement looks very similar to that of Horrocks 
(Theorem \ref{thm:horrocks} (1)). 

\begin{pbm}
Establish the relation between 
Yuzvinsky's and Horrocks' criteria for freeness. 
(More precisely, establish the relation between 
Yuzvinsky's lattice cohomology and sheaf cohomology on 
$\bP^n$.) 
\end{pbm}

Here we recover (slightly modified version of) 
Yuzvinsky's openness result for $\ell=3$ by 
using upper semicontinuity of exponents of $2$-multiarrangements. 
Similar to $\calM_{\ell,n}(L)$, we introduce the following set 
of arrangements which have prescribed characteristic polynomial. 
Let $f(t)\in\bZ[t]$. 
$$
\calC_{\ell,n}(f)
=
\{\A=(H_1, \dots, H_n)\in
(\bP^{\ell-1*})^n\mid 
H_i\neq H_j, \chi(\A, t)=f(t)\}. 
$$

\begin{theorem}
\label{thm:copen}
The set 
$$
\calC_{\ell,n}(f)^{\tiny \mbox{free}}=
\{\A\in\calC_{\ell,n}(f)\mid\mbox{$\A$ is free}\}
$$
is a Zariski open subset of $\calC_{\ell,n}(f)$. 
\end{theorem}
\begin{proof}
By Terao's factorization theorem, if $f(t)$ is not split, then 
$\calC_{\ell,n}(f)^{\tiny \mbox{free}}$ is empty. 
We may assume that 
$f(t)=(t-1)(t-d_1)(t-d_2)$. Fix $H_1\in\A$ and set 
$\exp(\A^{H_1}, \m^{H_1})=(d_1^{H_1}, d_2^{H_1})$. Then by 
Theorem \ref{thm:char3}, 
$|d_1^{H_1}-d_2^{H_1}|\geq |d_1-d_2|$ and 
$\A$ is free if and only if the equality holds. By the 
upper semicontinuity (Proposition \ref{prop:upper}) of 
the difference $\Delta(\A^{H_1}, \m^{H_1})=|d_1^{H_1}-d_2^{H_1}|$, 
the free locus $\{\A\mid \Delta<|d_1-d_2|+\frac{1}{2}\}$ is a 
Zariski open subset of $\calC_{\ell,n}(f)$. 
\end{proof}
Let $L$ be a poset, and $f(t)$ be the corresponding characteristic polynomial. 
Then $\calM_{\ell,n}(L)\subset \calC_{\ell,n}(f)$. Since 
$\calM_{\ell,n}(L)^{\tiny \mbox{free}}=
\calM_{\ell,n}(L)\cap \calC_{\ell,n}(f)^{\tiny \mbox{free}}$. 
We have obtained Yuzvinsky's 
openness result for $\ell=3$. 


We conclude this section with an observation. 
Lots of free arrangements which are not inductively free 
are rigid. It seems natural to ask whether or not the following 
(which is stronger than Terao conjecture) holds: 
\begin{equation}
\label{eq:rigid}
\{\mbox{Free arrangements}\}\subset
\{\mbox{Inductively free}\}
\cup
\{\mbox{Rigid}\}. 
\end{equation}

\subsection{Affine connection $\nabla$}
\label{subsec:conn}

\begin{defn}
Let $\delta, \theta\in\Der_V$. Set $\theta=\sum_if_i\partial_{x_i}$. 
Define $\nabla_\delta\theta\in\Der_V$ by 
$$
\nabla_\delta\theta=\sum_i(\delta f_i)\partial_{x_i}. 
$$
\end{defn}
It is easily seen that for any linear form $\alpha\in V^*$, 
\begin{equation}
\label{eq:nabla}
(\nabla_\delta\theta)\alpha=\delta(\theta\alpha). 
\end{equation}
Using this we have the following. 
\begin{prop}
\label{prop:deriv}
Let $(\A, \m)$ be a multiarrangement with 
$\m(H)>0, \forall H\in\A$. Let $\delta\in D(\A, \m)$ and 
$\eta\in\Der_V$. Then $\nabla_\eta\delta\in D(\A, \m-1)$. 
\end{prop}
\begin{proof}
By the assumption, we may write 
$\delta\alpha_H=\alpha_H^{\m(H)}F$. Then applying (\ref{eq:nabla}) we have 
$$
(\nabla_\eta\delta)\alpha_H=\eta(\alpha_H^{\m(H)}F)
=\m(H)\alpha_H^{\m(H)-1}\eta(\alpha_H)F+\alpha_H^{\m(H)}\eta (F), 
$$
which is divisible by $\alpha_H^{\m(H)-1}$. 
\end{proof}
The use of the connection $\nabla$ goes back to K. Saito 
\cite{sai-lin, sai-orb}. He studied discriminant in the Coxeter 
group quotient $V/W$. The space $V/W$ admits a degenerate metric 
induced from the $W$-invariant metric $I$ on $V$. The connection 
$\nabla$ is originally defined as the Levi-Civita connection 
for the degenerate metric. 
Since $I$ is 
flat on $V$, it is nothing but the connection above (see also 
\S\ref{sec:prim}). 
It has been gradually recognized that $\nabla$ is useful for the construction 
of various vector fields \cite{abe-exp, ay-quasi, st-double, 
ter-multi, ter-hodge, yos-prim}. 

Here we give a proof of 
Proposition \ref{prop:typical} (iii). 
\begin{prop}
Let $\A=\{H_1, \dots, H_n\}$ be a $2$-arrangement. Then 
the multiarrangement $(\A, 2)$ is free 
with exponents $(d_1, d_2)=(n,n)$. 
\end{prop}
\begin{proof}
Since $d_1+d_2=2n$, it is sufficient to show that 
there does not exist $\delta\in D(\A, 2)$ with $\pdeg\delta=n-1$. 
Suppose that it exists. Then by Proposition \ref{prop:deriv}, 
$\nabla_{\partial_{x_1}}\delta, 
\nabla_{\partial_{x_2}}\delta\in D(\A, 1)$ and 
$\pdeg\nabla_{\partial_{x_1}}\delta=\pdeg\nabla_{\partial_{x_2}}\delta=n-2$. 
Since $(\A, 1)$ is free with exponents $(1, n-1)$ and the degrees of 
$\nabla_{\partial_{x_1}}\delta$ and $\nabla_{\partial_{x_2}}\delta$ are 
smaller than $n-1$, they are multiples of the Euler 
vector field $\theta_E$ (Lemma \ref{lem:half}). 
We have expressions 
$\nabla_{\partial_{x_1}}\delta=F_1\cdot\theta_E, 
\nabla_{\partial_{x_2}}\delta=F_2\cdot\theta_E$ with $\deg F_1=\deg F_2=n-3$. 
On the other hand, 
$$
(\pdeg\delta)\cdot\delta=
\nabla_{\theta_E}\delta=
x_1\nabla_{\partial_{x_1}}\delta+
x_2\nabla_{\partial_{x_2}}\delta=
(x_1F_1+x_2F_2)\theta_E. 
$$
Hence $(x_1F_1+x_2F_2)\theta_E\alpha_H=
(x_1F_1+x_2F_2)\alpha_H$ is divisible by $\alpha_H^2$ for 
all $H\in\A$, equivalently, 
$x_1F_1+x_2F_2$ is divisible by 
$\prod_{i=1}^n\alpha_{H_i}$. However it contradicts 
$\deg(x_1F_1+x_2F_2)=n-2$. 
\end{proof}





\section{K. Saito's theory of primitive derivation}
\label{sec:prim}

Let $V=\bR^\ell$. Let $W$ be a finite reflection group which 
is generated by reflections in $V$ and acts irreducibly on $V$. 
The set of reflecting hyperplanes $\A(W)$ is called the Coxeter 
arrangement. 
There exists, unique up to a constant factor, a $W$-invariant 
symmetric bilinear form $I:V\times V\longrightarrow\bR$. The 
bilinear form $I$ induces a linear isomorphism 
$I:V^*\longrightarrow V$. Let $S=S(V^*)$ be the symmetric product. 
Since $\Der_V=S\otimes V$ and $\Omega_V=S\otimes V^*$, the map $I$ 
can be extended to an 
$S$-isomorphism $I:\Omega_V\longrightarrow\Der_V$. 

We first observe that a $W$-invariant vector field 
$\delta\in\Der_V^W$ is logarithmically tangent to $\A$. 
Indeed, let $\alpha_H\in V^*$ be a defining linear form 
of $H\in\A$ and $r_H\in W$ be the reflection with respect to 
$H$. Then $r_H(\alpha_H)=-\alpha_H$ and we have 
$r_H(\delta\alpha_H)=-\delta\alpha_H$. 
It is easily seen that if a polynomial $f\in S$ satisfies 
$r_H(f)=-f$, then $f$ is divisible by $\alpha_H$. Therefore 
$\delta\alpha_H$ is divisible by $\alpha_H$. Hence 
$\Der_V^W\subset D(\A)^W$. 

The ring $S^W$ of invariant polynomials is known to be 
isomorphic to a polynomial ring $\bR[P_1, P_2, \dots, P_\ell]$ 
(Chevalley \cite{che}). 
We can choose the polynomials $P_1, \dots, P_\ell$ to be homogeneous, 
with degrees 
$2=\deg P_1< \deg P_2\leq\dots\leq\deg P_{\ell-1}<\deg P_\ell$. 
The numbers $e_i=\deg P_i-1$, $i=1, \dots, \ell$ are 
called the exponents and $h=\deg P_\ell$ the Coxeter number. 
The Coxeter arrangement $\A$ is free. Furthermore, the basis 
of $D(\A)$ can be constructed explicitly by using basic invariants 
$P_1, \dots, P_\ell$. 

\begin{theorem}
\label{thm:coxeterfree}
(\cite{sai-log, sai-lin, sai-orb})
\begin{equation*}
\begin{split}
D(\A)^W &=\Der_V^W=\bigoplus_{i=1}^\ell S^W\cdot I(dP_i)\\
D(\A)&=\Der_V^W\otimes_{S^W}S=\bigoplus_{i=1}^\ell S\cdot I(dP_i). 
\end{split}
\end{equation*}
\end{theorem}
In particular, the Coxeter arrangement $\A$ is free with 
$\exp(\A)=(e_1, \dots, e_\ell)$. 
\begin{proof}
We shall give the proof of the second equality. From the above 
remarks, the inclusions 
\begin{equation}
\label{eq:supset}
D(\A)\supset \Der_V^W\otimes_{S^W}S\supset 
\bigoplus_{i=1}^\ell S\cdot I(dP_i), 
\end{equation}
are clear. Fix a coordinate system $(x_1, \dots, x_\ell)$. 
Recall that the Jacobian of the basic invariant 
$$
\Delta:=
\frac{\partial(P_1, \dots, P_\ell)}{\partial(x_1, \dots, x_\ell)}
=\prod_{H\in\A}\alpha_H, 
$$
is the product of linear forms of reflecting hyperplanes 
up to non-zero constant factors. 
Hence by Saito's criterion (Theorem \ref{thm:saitocri}), 
$I(dP_1), \dots, I(dP_\ell)$ form a basis of $D(\A)$. Thus 
the left hand side and right hand side in (\ref{eq:supset}) 
are equal. 
\end{proof}
Fix a system of basic invariants $P_1, \dots, P_\ell$ and 
a coordinate system $x_1, \dots, x_\ell$. Since 
$\deg P_i<\deg P_\ell$ ($i=1, \dots, \ell-1$), the rational 
vector field 
\begin{equation}
D=\frac{\partial}{\partial P_\ell}=
\frac{1}{\Delta}\det
\begin{pmatrix}
\frac{\partial P_1}{\partial x_1}&
\cdots&
\frac{\partial P_{\ell-1}}{\partial x_1}&
\frac{\partial}{\partial x_1}\\
\frac{\partial P_1}{\partial x_2}&
\cdots&
\frac{\partial P_{\ell-1}}{\partial x_2}&
\frac{\partial}{\partial x_2}\\
\vdots&\ddots&\vdots&\vdots\\
\frac{\partial P_1}{\partial x_\ell}&
\cdots&
\frac{\partial P_{\ell-1}}{\partial x_\ell}&
\frac{\partial}{\partial x_\ell}
\end{pmatrix}
\end{equation}
is uniquely determined up to constant factor, 
and it is also characterized 
by 
$$
DP_i=
\left\{
\begin{array}{cc}
1&\ i=\ell,\\
0&\ i\neq\ell. 
\end{array}
\right.
$$
The vector field $D$ is called the primitive vector field. 
\begin{theorem}
\label{thm:filt}
(\cite{sai-lin, sai-orb}) 
For every $W$-invariant vector field $\delta\in D(\A)^W$, 
there exists a unique vector field $\theta\in D(\A)^W$ such that 
$$
\nabla_D\theta=\delta. 
$$
We denote $\theta=\nabla_D^{-1}\delta$. 
\end{theorem}
Thus the operator $\nabla_D^{-1}$ acts on $D(\A)^W$. It induces 
a filtration, the so-called ``Hodge filtration'', 
\begin{equation}
\label{eq:hodgefilt}
\cdots
\nabla_D^{-2}D(\A)^W\subset
\nabla_D^{-1}D(\A)^W\subset
D(\A)^W. 
\end{equation}
The operator increases the contact order of the vector fields.

\begin{theorem}
\label{thm:coxetermulti}
(\cite{ay-quasi, st-double, ter-multi, ter-hodge, yos-prim})
Let $\A$ be a Coxeter arrangement with exponents $(e_1, \dots, e_\ell)$ and 
Coxeter number $h$. Let $\m:\A\longrightarrow\{0,1\}$ be a 
$\{0, 1\}$-valued multiplicity. 
\begin{itemize}
\item[(i)] 
For a positive integer $k$, we have 
\begin{equation*}
\begin{split}
D(\A, 2k+\m)&\simeq D(\A, \m)[-kh],\\ 
D(\A, 2k-\m)&\simeq \left(D(\A, \m)^\lor\right)[-kh]\simeq
\Omega^1(\A, \m)[-kh].
\end{split}
\end{equation*}
\item[(ii)] 
($\m\equiv 1$) 
The multiarrangement $(\A, 2k+1)$ is free with 
$\exp(\A, 2k+1)=(e_1+kh, \dots, e_\ell+kh)$. 
\item[(iii)] 
($\m\equiv 0$) 
The multiarrangement $(\A, 2k)$ is free with 
$\exp(\A, 2k)=(kh, kh, \dots, kh)$. 
\end{itemize}
\end{theorem}
In particular, 
the filtration (\ref{eq:hodgefilt}) is equivalent to the following. 
\begin{equation}
\label{eq:contact}
\cdots\subset
D(\A, 5)^W\subset
D(\A, 3)^W\subset
D(\A)^W. 
\end{equation}

\section{Weyl, Catalan and Shi arrangements}
\label{sec:appl}

In this section we consider a 
crystallographic Coxeter group (Weyl group) $W$. 
The reflecting hyperplanes are determined by a root system 
$\Phi\subset V^*$. We fix a positive system $\Phi^+\subset\Phi$. 
For a given $\alpha\in\Phi^+$ and $k\in\bZ$, define an affine hyperplane 
$H_{\alpha, k}$ by 
$$
H_{\alpha, k}=
\{\bm{x}\in V\mid \alpha(\bm{x})=k\}. 
$$
We consider the following type of arrangement 
$$
\A_\Phi^{[a, b]}=\{
H_{\alpha, k}\mid \alpha\in\Phi^+, k\in\bZ, a\leq k\leq b\}, 
$$
where $a\leq b$ are integers. (For example, see Figure \ref{fig:G2cat} for 
$\A_{G_2}^{[-1,1]}$.) 

\subsection{Freeness of Extended Catalan and Shi arrangements}
\label{subsec:er}

The next result was originally conjectured by Edelman-Reiner 
\cite{ede-rei}. 

\begin{theorem}
\label{thm:er}
(\cite{yos-char}) Let $k$ be a nonnegative integer. 
\begin{itemize}
\item[(i)] 
The cone $c\A_\Phi^{[-k,k]}$ of the extended Catalan arrangement 
$\A_\Phi^{[-k,k]}$ is free with exponents $(1, e_1+kh, \dots, e_\ell+kh)$. 
\item[(ii)] 
The cone $c\A_\Phi^{[1-k,k]}$ of the extended Shi arrangement 
$\A_\Phi^{[1-k,k]}$ is free with exponents $(1, kh, kh, \dots, kh)$. 
\end{itemize}
\end{theorem}
\begin{proof}
The proof is done by induction on the rank $\ell$ of the root system $\Phi$. 
First one can check for the case $\ell=2$, $\Phi=A_2, B_2, G_2$ (using 
Theorem \ref{thm:char3} and Theorem \ref{thm:coxetermulti}, 
or proving inductive freeness). For 
$\ell\geq 3$, consider the restriction of 
$c\A_\Phi^{[-k,k]}$ (resp. $c\A_\Phi^{[1-k,k]}$) to the hyperplane at 
infinity $H_0$ and apply Theorem \ref{thm:char4}. 
The multirestriction 
$((c\A_\Phi^{[-k,k]})^{H_0}, \m^{H_0})$ 
(resp. $((c\A_\Phi^{[1-k,k]})^{H_0}, \m^{H_0})$) is 
equal to the multiarrangement 
$(\A, 2k+1)$ (resp. $(\A, 2k)$). Thus the second condition in 
Theorem \ref{thm:char4} is verified by 
Theorem \ref{thm:coxetermulti} (ii) (resp. (iii)). The 
localization of $c\A_\Phi^{[-k,k]}$ at $x\in H_0\setminus\{0\}$ is 
a direct sum of Coxeter arrangements of lower ranks. Hence 
the first condition in Theorem \ref{thm:char4} is verified by 
the inductive assumption. 
\end{proof}

Using Terao's factorization theorem (Theorem \ref{thm:tfact}), we 
have the following. 

\begin{corollary}
\label{cor:er}
\begin{itemize}
\item[(i)] 
$\chi(\A_\Phi^{[-k,k]}, t)=\prod_{i=1}^\ell(t-e_i-kh)$. 
\item[(ii)] 
$\chi(\A_\Phi^{[1-k,k]}, t)=(t-kh)^\ell$. 
\end{itemize}
\end{corollary}
In the above corollary, (i) has been proved by Athanasiadis 
\cite{ath-gen} by a purely combinatorial method. Edelman-Reiner 
\cite{ede-rei} and Headley \cite{hea-shi} proved (ii) for some 
special cases (type $A$ and the case $k=1$). However as far as we know, 
the combinatorial proof for (ii) is not known.

\subsection{Beyond free arrangements}

There are several conjectures on the characteristic polynomials 
$\chi(\A_\Phi^{[a,b]}, t)$. 
\begin{conj}
\label{conj:rh}
(``Riemann hypothesis'' by Postnikov-Stanley, \cite{ps-def}) 
If $0\leq a<b$ are integers, then all roots of 
the characteristic polynomial 
$\chi(\A_\Phi^{[-a,b]}, t)$ have the same real part 
$\frac{(a+b+1)h}{2}$. 
\end{conj}
This conjecture has been verified for types $ABC$ and $D$ by 
Athanasiadis \cite{ath-lin}. We also note that for the parameters 
$b=a+1$, it is a special case of Theorem \ref{thm:er} (ii). Generally 
it is still an open problem. Conjecture \ref{conj:rh} implies that 
the roots of the characteristic polynomial $\chi(\A_\Phi^{[-a,b]}, t)$ 
sit on the line of complex numbers whose real part is 
$\Real=\frac{(a+b+1)h}{2}$, which concludes the following 
nontrivial property of the characteristic polynomial. 
\begin{conj}
\label{conj:fe}
(``Functional Equation'' by Postnikov-Stanley, \cite{ps-def}) 
If $a, b$ are integers such that $-1\leq a\leq b$ (except for 
$(a,b)=(-1,0)$ and $(-1,-1)$), then the characteristic 
polynomial satisfies 
\begin{equation}
\label{eq:fe}
\chi(\A_\Phi^{[-a,b]}, (a+b+1)h-t)=
(-1)^\ell \chi(\A_\Phi^{[-a,b]}, t). 
\end{equation}
\end{conj}
Note that the ``Functional Equation'' is true when $a=b\geq 0$. 
Indeed, in this case $\chi(\A_\Phi^{[-a,a]},t)=\prod_{i=1}^\ell
(t-e_i-ah)$. The relation (\ref{eq:fe}) is equivalent to 
the relation so-called duality of exponents: 
\begin{equation}
\label{eq:dual}
e_i+e_{\ell+1-i}=h, 
\end{equation}
$i=1, \dots, \ell$. Thus the ``Functional Equation'' can be 
considered as a generalization of the duality of exponents. 

The following is also observed in the work by Athanasiadis 
\cite{ath-lin, ath-gen}. 

\begin{conj}
\label{conj:h-shift}
If $a, b$ are integers such that $-1\leq a\leq b$ (except for 
$(a,b)=(-1,0)$ and $(-1,-1)$), 
then the characteristic 
polynomial satisfies 
\begin{equation}
\label{eq:hshift}
\chi(\A_\Phi^{[-a-1,b+1]},t)=
\chi(\A_\Phi^{[-a,b]},t-h). 
\end{equation}
\end{conj}
Except for $[a,b]=[-k,k]$ and $[1-k,k]$, the characteristic polynomial 
$\chi(\A_\Phi^{[a,b]}, t)$ can not be decomposed into linear terms. So 
the cone $c\A_\Phi^{[a,b]}$ is no more free. 
The simplest such example may be $\Phi=A_3$ ($h=4$) 
with $(a,b)=(-1,1)$. 
More explicitly, after change of coordinates, 
\begin{equation*}
\begin{split}
Q(\A_{A_3}^{[1,1]})&=(x-1)(y-1)(z-1)(x+y-1)(y+z-1)(x+y+z-1)\\
Q(\A_{A_3}^{[0,2]})&=
\prod_{k=0}^2
(x-k)(y-k)(z-k)(x+y-k)(y+z-k)(x+y+z-k). 
\end{split}
\end{equation*}
Then 
$\chi(\A_{A_3}^{[1,1]},t)=(t-2)(t^2-4t+7)$ and 
$\chi(\A_{A_3}^{[0,2]},t)=(t-6)(t^2-12t+39)$. 
It is easily seen 
that roots have the real part $2$, respectively $6$, and 
$
\chi(\A_{A_3}^{[0,2]},t)=
\chi(\A_{A_3}^{[1,1]},t-4)$. 

It seems to be interesting 
to investigate these conjectures through the module $D(\A)$ 
of logarithmic vector fields. 

\begin{pbm}
Prove the above conjectures by using $D(\A)$. Is it possible to 
refine these conjectures in terms of algebraic/geometric structures 
of the module of logarithmic vector fields? 
\end{pbm}

\begin{rem}
Recall that 
$D_0(\A_\Phi^{[a,b]})\simeq D(\A_\Phi^{[a,b]})/S\cdot\theta_E$. 
In the lecture in Pau (June 2012), 
the author asked whether or not 
if the following isomorphisms hold  
\begin{equation*}
\begin{split}
D_0(c\A_\Phi^{[-a,b]})^\lor&\simeq
D_0(c\A_\Phi^{[-a,b]})[(a+b+1)h]\\
D_0(c\A_\Phi^{[-a-1,b+1]})&\simeq
D_0(c\A_\Phi^{[-a,b]})[-h], 
\end{split}
\end{equation*}
which induce Conjecture \ref{conj:fe} and Conjecture 
\ref{conj:h-shift}, respectively via Solomon-Terao's formula 
Theorem \ref{thm:stf} (see \cite{ay-quasi} for details). 
These seem to be strongly supported 
by Theorem \ref{thm:coxetermulti}. 
However Professor D. Faenzi 
pointed out to us (October 2012) that the first isomorphism 
$D_0(c\A_\Phi^{[-a,b]})^\lor\simeq
D_0(c\A_\Phi^{[-a,b]})[(a+b+1)h]$ does not hold at least 
for some cases. 
\end{rem}




\end{document}